\documentclass[11pt,a4paper]{article}
\usepackage[utf8]{inputenc}
\usepackage{amsmath, amsthm}
\usepackage{amsfonts}
\usepackage{amssymb}
\usepackage{graphicx}
\usepackage[hidelinks]{hyperref}
\usepackage{lmodern,microtype} 
\usepackage[T1]{fontenc} 
\usepackage[left=2.5cm,right=2.5cm,top=2.5cm,bottom=2.5cm]{geometry}
\usepackage{mathrsfs}
\usepackage{mathtools}
\usepackage{xcolor}
\usepackage{enumitem}
\usepackage{comment}
\usepackage{xfrac}

\usepackage{titlesec} 

\titleformat{\section}[hang]{\centering \large\sc}{\thesection}{1em}{}[]

\titleformat{\subsection}[hang]{ \bfseries}{\thesubsection}{1em}{}[]

\titleformat{\subsubsection}[runin]{\bfseries}{}{0pt}{}[.]

\newtheorem{theorem}{Theorem}
\newtheorem{lem}{Lemma}
\newtheorem{prop}{Proposition}
\newtheorem*{remark}{Remark}
\newtheorem*{defin}{Definition}

\usepackage{csquotes}
\usepackage[
url=false,
backend=biber,
style=numeric-comp,
giveninits=true,
doi=true,
backref=false,
date=year,
maxbibnames=99,
eprint=false
]{biblatex}

\bibliography{biblio}

\usepackage{authblk}

\title{Sharp lifespan results for plane-symmetric fluids on decelerated spacetimes}

\author{Maximilian Ofner\thanks{Department of Pure Mathematics and Mathematical Statistics, University of Cambridge, Wilberforce Road, CB3 0WB Cambridge, United Kingdom. \emph{Correspondence: mo578@cam.ac.uk}}\hspace{0.4em} and   Todd Oliynyk\thanks{School of Mathematics, Monash University, 9 Rainforest Walk, VIC 3800 Melbourne, Australia.}}

\date{}

\begin{document}

    \maketitle

    \begin{abstract}
        In this article, we analyze plane-symmetric solutions to the relativistic Euler equations on spatially flat spacetimes with decelerated expansion. We consider a linear equation of state $p(\rho)=K\rho$ and power law inflation $a(t)=t^{\alpha}$, where $0<K<\frac{1}{3}$ and $0\leq\alpha<1$.  Previous results imply nonlinear stability of the homogeneous and isotropic solution above an expansion threshold $\alpha_\text{crit}=\frac{2}{3(1-K)}$. We prove that at and below this critical expansion rate there exist arbitrarily small data that exhibit finite-time singularity formation of the shock type, i.e., gradient blowup while the solution stays bounded. For this type of data, we establish lifespan estimates that are polynomial below and exponential at the critical threshold. In addition, we prove a theorem based on energy estimates that provides lower bounds on the lifespan of plane-symmetric solutions generated from small initial data and generalizes to solutions without symmetry in $3+1$-dimensions.
    \end{abstract}
    
    \section{Introduction}\label{sec:Introduction}

    Spatially homogeneous and isotropic spacetimes that expand towards the future form the foundation of modern cosmology. The observed large-scale properties of our Universe are well described by such spacetimes, while deviations from exact homogeneity and isotropy are treated as small perturbations. Understanding the global and asymptotic properties of these spacetimes, together with the matter they contain, is therefore central to developing a rigorous mathematical foundation for cosmology.
    
    Relativistic perfect fluids constitute one of the most important classes of matter models in cosmology. Their dynamics are governed by the \textit{relativistic Euler equations}:
    \begin{equation}\label{eq:releuler}
        \begin{aligned}
            \nabla_\mu T^{\mu\nu}&=0,\\
            T^{\mu\nu}&=(p+\rho)u^{\mu}u^{\nu}+pg^{\mu\nu},\\
            g(u,u)&=-1,
        \end{aligned}
    \end{equation}
    where $g$ is the spacetime metric, $\nabla$ is the Levi-Civita connection of $g$, $p$ is the fluid pressure, $\rho$ is the proper energy density of the fluid and $u$ is the fluid four-velocity. To close this system, we assume, as is standard in cosmology, a 
    linear barotropic equation of state of the form
    \begin{equation} \label{eos}
        p(\rho)=K\rho. 
    \end{equation}
    The parameter $K$ represents the square of the fluid sound speed and, for physical reasons, is typically restricted to the interval $[0,1]$. Fluids with $K=0$, $K=\frac{1}{3}$, and $K=1$ are referred to as \textit{dust}, \textit{radiation fluids}, and \textit{stiff fluids}, respectively. 

    In cosmology, relativistic fluids are coupled to the gravitational field through Einstein's equations. However, due to the complexity of the resulting coupled system, important insights into the behaviour of relativistic fluids in cosmological settings can be gained by first considering their evolution on a \textit{prescribed} cosmological spacetime. This allows the fluid dynamics to be studied in isolation from the complications arising from the backreaction of the fluid on the spacetime geometry. This is the perspective we take in this article. 

    In the following, we restrict our attention to spatially homogeneous spacetimes of the form 
    \begin{equation}\label{eq:spacetime}
        \left(M=[t_{0},\infty)\times \mathbb{T}^{3},g=-dt^2+a(t)^2\delta_{ij}dx^{i}dx^{j}\right).
    \end{equation} 
    A distinguishing feature of these spacetimes is that their spatial sections are flat and compact. The function $a(t)$ appearing the metric $g$ is known as the \textit{scale factor} and it satisfies $a(t)>0$ and $a'(t)>0$ for all $t\in [t_0,\infty)$. These spacetimes can be further classified according to whether the expansion is \textit{accelerated}, \textit{linear}, or \textit{decelerated}, corresponding, respectively, to $a''(t)>0$, $a''(t)=0$, or $a''(t)<0$. 
    
    The future stability and asymptotic behavior of solutions to the relativistic Euler equations on spacetimes undergoing \textit{decelerated power-law expansion} were first investigated in \cite{fajman2025cqg,fajman2025arxiv} for the sub-radiative equations of state with 
    \begin{equation} \label{eq:sub-rad} 
    K\in \left(0,\frac{1}{3}\right).
    \end{equation}
    The spacetimes considered there were of the form \eqref{eq:spacetime}, with scale factor
    \begin{equation}\label{eq:pow-law}
    a(t)=t^{\alpha}, \qquad \alpha\in(0,1).
    \end{equation}
    The boundary cases of dust and radiation fluids on decelerating backgrounds were studied earlier in \cite{saisai2020,speck2013}.

    The primary innovation of the works \cite{fajman2025cqg,fajman2025arxiv} was the identification of the curve
    \begin{equation}\label{eq:criteq}
    K=1-\frac{2}{3\alpha}
    \end{equation}
    that separates the $(\alpha,K)$ region $(0,1)\times (0,\frac{1}{3})$ into stable and unstable regions. In the \textit{stable region}, defined by
    \begin{equation*}
    K<1-\frac{2}{3\alpha},
    \end{equation*}
    it was rigorously established that \textit{orthogonal}\footnote{These are homogeneous solutions with vanishing spatial velocity.} solutions to the relativistic Euler equations, see \eqref{eq:background}, are nonlinearly stable to the future. In the complementary region, consisting of the critical curve \eqref{eq:criteq} together with
    \begin{equation*}
    K> 1-\frac{2}{3\alpha},
    \end{equation*}
    which we collectively refer to as the \textit{unstable region},
    a detailed numerical study of solutions to the $\mathbb{T}^2$-symmetric relativistic Euler equations was carried out. The results provided compelling evidence that, in this region, small $\mathbb{T}^2$-symmetric perturbations of spatially homogeneous and isotropic solutions to the relativistic Euler equations develop singularities in finite time. Moreover, these numerical studies indicated the following dependence of the singularity formation time $t_{\textrm{sng}}$ on the size of the initial perturbation $\epsilon$:
    \begin{equation} \label{t-sing}
    t_{\textrm{sng}}\sim\begin{cases} \epsilon^{\frac{(3K+2)(1-K)}{2-3\alpha(1-K)}} & \text{if $ K> 1-\frac{2}{3\alpha}$} \\
    e^{\frac{e^b}{\epsilon}}& \text{if $ K=1-\frac{2}{3\alpha}$}\end{cases}.
    \end{equation}
    For the boundary case $K=0$, corresponding to dust, singularity formation was established in \cite{saisai2020} for $0<\alpha\leq \frac{1}{2}$, while future stability was established in \cite{speck2013} for $\frac{1}{2}<\alpha<1$. For the other boundary case, $K=\frac{1}{3}$, corresponding to radiation, \cite{speck2013} established singularity formation for all $\alpha \in (0,1)$. Alternative proofs for singularity formation in radiation and dust were given in \cite{fajman2025arxiv}.

    The main result of this article is a rigorous proof of the singularity formation observed numerically in \cite{fajman2025cqg} for parameter values $(\alpha,K)$ in the unstable region. The precise statement of our singularity formation result is presented below in Theorem \ref{blowup}. Together the results of this article and those of \cite{saisai2020,speck2013,fajman2025arxiv} provide a comprehensive picture of the future behavior of small perturbations of orthogonal solutions to the relativistic Euler equations over the parameter values $(\alpha,K)\in (0,1)\times [0,\frac{1}{3}]$. 
    
    \subsection{Prior and related results}
    \subsubsection{Accelerated expansion}

    The study of the stability properties of inhomogeneous fluids on expanding spacetimes was initiated in \cite{brauer1994}, where the stabilizing effect of spacetime expansion on fluid evolution was first observed. In particular, the authors rigorously showed that accelerated spacetime expansion can suppress singularity formation. They also provided an example of a spacetime undergoing decelerated expansion on which dust develops a singularity in finite time.

    Building on the techniques introduced in \cite{ringstrom2008}, the stability results of \cite{brauer1994} were subsequently generalized in \cite{speck2012,rodnianski2013}, where it was shown that, for linear equations of state \eqref{eos} with $0<K<\frac{1}{3}$, FLRW solutions to the Einstein--Euler equations with a positive cosmological constant and undergoing accelerated expansion are nonlinearly stable toward the future. This stability result was later extended to the endpoint cases of dust and radiation, corresponding to $K=0$ and $K=\frac{1}{3}$, in \cite{hadzic2015} and \cite{lubbe2013}, respectively. Alternative approaches to establishing future stability were developed in \cite{oliynyk2016,liu2018ahp,liu2018, friedrich2017}, while analogous stability results for fluids with nonlinear equations of state and other related matters model were established in \cite{wei2018,gong2024,liu2021,lefloch2021,Lubbe2024}.

    The situation is markedly different for super-radiative linear equations of state, $\frac{1}{3}<K\leq 1$. For $\frac{1}{3}<K<1$, it has been shown that \textit{tilted}\footnote{These are homogeneous solutions with non-vanishing spatial velocity.} solutions of the relativistic Euler equations are nonlinearly stable to the future on spacetimes undergoing accelerated expansion \cite{marshall2023,oliynyk2021siam}. This result was subsequently generalized in \cite{fournodavlos2026,fournodavlosAHP2026}, within the more restrictive parameter range $\frac{1}{3}<K<\frac{5}{7}$, to the coupled Einstein-Euler system with a positive cosmological constant.

    In contrast, orthogonal solutions to the relativistic Euler equations in the super-radiative regime are unstable, despite the accelerated expansion of the background spacetime. For $\frac{1}{3}<K<1$, this instability was rigorously established in \cite{oliynyk2024cmp}, where it was shown that the fractional density gradient develops sharp features and becomes unbounded at future timelike infinity. This behavior was first conjectured in \cite{rendall2004} to occur for perturbations of FLRW solutions to the Einstein-Euler equations with a positive cosmological constant and is known as the \textit{Rendall instability}. Numerical solutions to the Einstein-Euler equations exhibiting this instability have been constructed in \cite{beyer2023}. At the stiff-fluid endpoint, $K=1$, a different instability occurs: it was shown in \cite{Fournodavlos2022} that irrotational perturbations of FLRW solutions to the Einstein-Euler equations with a positive cosmological constant develop singularities in finite where the fluid four-velocity fails to remain timelike.

    \subsubsection{Linear expansion} For sub-radiative linear equations of state \eqref{eos}, with $(0<K<\frac{1}{3})$, orthogonal solutions the irrotational relativistic Euler equations were shown in \cite{fajman2021cmp} to be nonlinearly stable toward the future on spacetimes undergoing linear expansion. This result complemented the earlier work \cite{speck2013}, which established finite-time shock formation for a radiation fluid $(K=\frac{1}{3})$. The stability results of \cite{fajman2021cmp} were subsequently extended in \cite{fajman2024arma,fajman2024imrn} to the coupled Einstein-Euler system, while simultaneously removing the assumption of irrotationality.  

    \subsubsection{Decelerated expansion} In addition to the articles discussed above, we point out the following recent stability results for solutions to the Euler and related equations on spacetimes undergoing decelerated expansion. We begin with the article \cite{FMOOW25b}, in which the stability results of \cite{fajman2025cqg,fajman2025arxiv} are generalized to include coupling to Newtonian gravity. Interestingly, stability is shown to hold for $(\alpha,K)\in (\frac{2}{3},1)\times(0,\frac{1}{3})$,
    which differs qualitatively from the stability region established in \cite{fajman2025cqg,fajman2025arxiv}.
    
    On spacetimes of the form \eqref{eq:spacetime} undergoing decelerated power-law expansion, the nonlinear future stability of solutions to the massless Boltzmann equation was established in \cite{strain2026}. Related stability results were obtained for expanding FLRW solutions to the spherically symmetric Einstein-massless Vlasov system in \cite{taylor2024jmp}. Finally, for FLRW solutions with compact spatial slices undergoing decelerated expansion, nonlinear future stability was established for the Einstein-scalar field and Einstein-Euler equations in \cite{bernhardt2025} and \cite{bernhardt2026}, respectively. In addition, \cite{Marshall2025} presents numerical evidence of instability of FLRW solutions featuring decelerated expansion to the Einstein-Euler equations in Gowdy symmetry.

    In spatially non-compact spacetimes of the type \eqref{eq:spacetime}, i.e., $M=[t_0,\infty)\times \mathbb{R}^3$, with decelerated expansion the dispersive effects of the wave equation have been studied recently in \cite{haghshenas2026,costa2029,Natario2023,rossetti2025,rossetti2026}. 
    
    \subsection{Main results}

    The rigorous statements of the main results of this article is described in Theorem \ref{blowup} and \ref{energyest} below. 
    
	\begin{theorem}\label{blowup}
		Let $s\geq 2$, $0 <  \alpha< \alpha_{\text{crit}}=\frac{2}{3(1-K)} $, and $ t_{0}\geq 1 $. There exists a sequence of initial data $((\mathring{R}_+,\mathring{R}_{-})_n)_{n\in \mathbb{N}} \subset C^{\infty}(S^{1})$, $ (0,0)\neq (\mathring{R}_+,\mathring{R}_{-})_n $ for all $n\in \mathbb{N}$, at $ t_{0} $ satisfying
		 \begin{equation*}
		 	\lim_{n\to \infty}\|(\mathring{R}_+,\mathring{R}_{-})_n\|_{H^{s}}=0, \qquad \delta_n \coloneqq  \|(\partial\mathring{R}_+,\partial\mathring{R}_{-})_n\|_{L^\infty} \xrightarrow[]{n\to\infty}0,
		 \end{equation*}
      such that the solutions $((R_+,R_{-})_n)_{n\in \mathbb{N}}$ of the system \eqref{eq:chareom} launched by $((\mathring{R}_+,\mathring{R}_{-})_n)_{n\in \mathbb{N}} $ have the following properties:
       \begin{enumerate}[label=(\roman*)]
           \item \label{thm1}The classical solution terminates at finite time $t_{\delta_n}$.
           \item \label{thm2}The gradient blows up, i.e.,
            \begin{equation*}
			\lim_{t\to t_{\delta_n}}\|(\partial R_+,\partial R_-)\|_{L^{\infty}}(t)=\infty.
		\end{equation*}
        \item \label{thm3}The amplitude remains bounded, i.e.,
        \begin{equation*}
			\sup_{t\in (t_0, t_{\delta_n})}\|(R_+,R_-)_n\|_{L^{\infty}}(t)<\infty.
		\end{equation*}
        \item There exists constants $c_1,c_2>0$ independent of $n$ such that the lifespan $t_{\delta_n}$ of the solution $(R_+,R_{-})_n$ satisfies
        \begin{equation*}
            c_1\delta_n^{-\frac{1}{1-\frac{3}{2}\alpha(1-K)}}\leq t_{\delta_n} \leq  c_2\delta_n^{-\frac{1}{1-\frac{3}{2}\alpha(1-K)}}.
        \end{equation*}
       \end{enumerate}
       Now restrict to $s=2$ and let $\alpha=\alpha_{\text{crit}}$. Then \ref{thm1},\ref{thm2}, and \ref{thm3} remain true, while there exist constants $c_3,c_4>0$ and $c_5,\tilde{c}_5>0$ independent of $n$ so that
       \begin{equation}\label{eq:blowupexp}
           c_3e^{\frac{c_5}{\delta_n}}\leq t_{\delta_n}\leq  c_4e^{\frac{\tilde{c}_5}{\delta_n}}.
       \end{equation}
	\end{theorem}

    \begin{remark}[Blowup of the physical variables]
        The data constructed in the proof of Theorem \ref{blowup} feature blowup in the derivative $\partial_x R_+$, while $R_+,R_-$, and $\partial_x R_-$ remain bounded. Inspecting the definition of the Riemann invariants \eqref{eq:riemann} as well as the transformation to physical variables \eqref{eq:Lv}, a straightforward computation implies that $\partial_x u^{x}$ as well as $\partial_x \rho$ blow up. Scenarios like this a often referred to as \emph{shock formation}. 
    \end{remark}

    \begin{remark}[Comparison with numerical data]
        Note that the lifespan results in Theorem \ref{blowup} are close to the ones predicted in \eqref{t-sing} taken from \cite{fajman2025cqg}, as the numerator in the exponent in \eqref{t-sing} is approximately $2$ for $0<K<\frac{1}{3}$. We attribute this difference to the numerical fitting process. 
    \end{remark}
    
    \begin{theorem}\label{energyest}
    Suppose that $s\geq 2$ and $\alpha < \alpha_\text{crit}$. Then there exist constants $\epsilon_{0}, c_{0}, c_{1}>0$ such that if $\|U_{0}\|_{H^{s}}\leq \epsilon$ for $\epsilon\leq \epsilon_{0}$ the unique solution $U$ with initial data $U(t_{0},\cdot)=U_{0}$ can be extended at least until time 
    \begin{equation}\label{eq:t1bound}
        t_{1}=c_{1}\epsilon^{-\frac{1}{1-\frac{3}{2}\alpha (1-K)}}.
    \end{equation}
    Furthermore, it satisfies the bound $\sup_{t_{0}\leq t<t_{1}} \|U(t,\cdot)\|_{\tilde{H}^{s}}<c_{0}\epsilon$. If $\alpha=\alpha_\text{crit}$ instead, there exist constants $c_2,c_3>0$ such that the solution can be extended to 
    \begin{equation}\label{eq:t1boundexp}
        t_{1}=c_{2}e^{\frac{c_3}{\epsilon}}.
    \end{equation}
    \end{theorem}
    \begin{remark}
        Note that the lower bounds in \eqref{eq:t1boundexp} in \eqref{eq:blowupexp} do not agree. The reason stems from the fact that for the sequence of data constructed in the proof of Theorem \ref{blowup}, the $H^2$-norm does not agree with its $W^{1,\infty}$-norm. If the former tends to zero as $\delta$, the latter does so as $\delta^{\frac{1}{2}}$. 
    \end{remark}
    
    \subsection{Overview of this article}

	\begin{itemize}
        \item Section \ref{sec:Introduction} provides a short introduction to the contents of this articles and an overview of recent developments. In addition, it rigorously states the main results given by Theorem \ref{blowup} and \ref{energyest}.
        \item Section \ref{sec:not} introduces the necessary notations and definitions. 
        \item Section \ref{sec:plane} introduces the relativistic Euler equations in plane-symmetry. This system admits a homogeneous and isotropic solution, the density of which decays in time. In expansion-normalized fluid variables $(L,v)$ the canonical solution becomes trivial, i.e., $(L,v)=(0,0)$. We are ultimately interested in the long-time properties of data close to this trivial solution. To study these solutions, we introduce \emph{Riemann invariants} $R_\pm$, which allow us to treat parts of these equations as an ODE. 
        \item Section \ref{sec:riccati} is dedicated to deriving an ODE-like evolution equation for the spatial derivative $\partial_xR_+$. In particular, we derive an equation akin to 
        \begin{equation*}
            \frac{d}{dt}(\partial_xR_+)(t,\varrho_+(t))\sim A_2 (\partial_xR_+)^2(t,\varrho_+(t))+A_1(\partial_xR_+)(t,\varrho_+(t))+A_0,
        \end{equation*}
        where $\varrho_+$ is a characteristic of $R_+$. The coefficients $A_{0,1,2}$ depend on the Riemann invariants $R_\pm$ and time but not their derivatives. The quadratic nature of this ODE, which is often referred to as \emph{Riccati equation}, can force blowup of its solution given suitable control on $A_{0,1,2}$. 
        \item Section \ref{sec:est} proves a global bound of the Riemann invariants. Supposing that the image of the data lies within a set
        \begin{equation*}
            D_{\mu}= \{(x,y)\in \mathbb{R}^{2}\vert \max(|x|,|y|)< \mu\},
        \end{equation*}
        we prove that the classical solution cannot leave $D_\mu$ as long as it exists. This is then used to establish bounds on key quantities that appear in $A_{0,1,2}$. 
        \item Section \ref{sec:thm1} is dedicated to the proof of Theorem \ref{blowup}. The previously established control of the Riemann invariants and, subsequently, the coefficients $A_{0,1,2}$ allows us to give an estimate of the evolution of the derivative of one of the Riemann invariants. In the subcritical case, $\alpha<\alpha_{\text{crit}}$, constructing a sequence of initial data is straightforward, as data $\mathring{R}_+$ satisfying
        \begin{equation*}
            \|\mathring{R}_+\|_{L^{\infty}} \sim \|\partial_x\mathring{R}_+\|_{L^{\infty}}
        \end{equation*}
        are sufficient to generate blowup. Tracking the evolution of the spatial derivative along the characteristic of maximal initial compression yields an upper bound on the lifespan. A lower bound can be established in a similar way. 

        In the critical case, we require additional compression of the initial data, such that
        \begin{equation*}
            \|\mathring{R}_+\|_{L^{\infty}} \sim \|\partial_x\mathring{R}_+\|_{L^{\infty}}^2.
        \end{equation*}
        
        As a consequence, the family of constructed data only tends to $0$ in the low regularity class $H^{2}$. 
        \item Lastly, in section \ref{sec:thm2}, we prove a lower bound for the lifespan of arbitrary small-data solutions based on energy methods. Although this is not sharp for the sequence constructed at the critical line, it is much more robust and generalizes to higher dimensions. 
	\end{itemize}

    \subsection{Declarations and acknowledgments}
    
    \subsubsection{Acknowledgments} M.O.~acknowledges support by the Austrian Science Fund (FWF) grant \emph{Cosmological Fluid Models Beyond the Critical Regime} 10.55776/J4982. The authors acknowledge the hospitality of the ICMS Edinburgh, where parts of the present research article were created during the program \emph{PDE perspectives in gravitational theory, September 2026}. 

    \subsubsection{Declaration of Usage of AI tools} AI tools (in particular \emph{GPT-5.6} and \emph{GPT-6} by \emph{OpenAI}) were used for checking calculations and searching pre-existing literature. Neither mathematics nor text in this article was generated by AI and the authors assume responsibility for all content. 
        
    \section{Notation and definitions}\label{sec:not}

    We assume the spacetime \eqref{eq:spacetime} to be equipped with the standard (almost) global coordinates $(t,x,y,z)$. We will denote indices for spatial degrees of freedom running in $\{1,2,3\}$ with roman letters $i,j,k, \dots $, while Greek letters $\mu,\nu,\xi, \dots$ are assumed to run in $\{0,1,2,3\}$ and are therefore spacetime indices. A spatial derivative $\partial_x f(\cdot,x)$ of a function in the spacetime variables $(t,x)$ will often simply be denoted as $\partial f$. 
    
    Furthermore, to avoid confusion, we make a strict distinction between \emph{slot-derivatives}, $\partial_1 f$ of a function $f$ in two variables and a \emph{coordinate-derivative} $\partial_x f(t,x)$. Hence, $(\partial_2f)(t,x)=\partial_x f(t,x)$. 
    
    For a definition and elementary properties of Sobolev spaces on compact manifolds, we refer to \cite{hebey1999}. 
    
    \begin{defin}[Plane-symmetry]\label{planesymmetry}
        We refer to a solution $(u^{i},\rho)$ of \eqref{eq:releuler} as \emph{$\mathbb{T}^{2}$-} or \emph{plane-symmetric}, if $$u^{y}=u^{z}=\partial_y \rho=\partial_z \rho=\partial_yu^{x}=\partial_z u^{x}=0.$$ 
    \end{defin}

    \begin{remark}
        Note that plane-symmetry propagates and, in particular, plane-symmetric data launches plane-symmetric solutions. This can easily be seen, e.g., inpecting the equations of motion as presented in \cite{fajman2025arxiv}. 
    \end{remark}

    \begin{defin}[Big-O notation]
        Let $\epsilon_{0}>0$ and $f,g\in C((0,\epsilon_{0}); \mathbb{R}_{+})$ with $\lim_{x\to0^{+}}f(x)=\lim_{x\to0^{+}}g(x)=0$  and $g>0$. Then we define
        \begin{equation*}
            f=\mathcal{O}(g) \iff \limsup_{x\to 0^{+}} \left|\frac{f(x)}{g(x)} \right|<\infty.
        \end{equation*}
        If instead $\lim_{x\to 0^{+}}f(x)=C$ with $C\in \mathbb{R}$, we define 
        \begin{equation*}
            f=C+\mathcal{O}(g) \iff |f-C|= \mathcal{O}(g).
        \end{equation*}
        We tacitly extend this concept to elementary manipulations. 
    \end{defin}

    \begin{defin}[Comparison]
        Given two functions $f,g\in C(M;\mathbb{R})$ we will often write $f<g$ meaning that for all $x\in M$, $f(x)< g(x)$. Furthermore, we will often write $a\lesssim b$ meaning there exists a constant $C>0$, such that $a\leq C b$ for all suitable arguments of $a$ and $b$ (real numbers, functions, etc.). In addition we write 
        \begin{equation*}
            a \simeq b \iff a\lesssim b \text{ and } b\lesssim a.
        \end{equation*}
    \end{defin}

    \section{The relativistic Euler equations in plane-symmetry}\label{sec:plane}

	\subsection{The background solution}
	The system \eqref{eq:releuler} admits a family of non-trivial, homogeneous, and isotropic solution, given explicitly by
    \begin{equation}\label{eq:background}
        \bar{u}^{\mu}(t,\cdot)=\delta^{\mu}_{0}, \qquad \bar{\rho}(t,\cdot)=\rho_0 \left(\frac{t}{t_0}\right)^{-3\alpha(1+K)},
    \end{equation}
    where $\rho_0\in \mathbb{R}_{+}$. In particular, this solution is plane-symmetric in the sense of Definition \ref{planesymmetry}. 

    We will, however, study the relativistic Euler equations in the fluid coordinates 
    \begin{equation}\label{eq:Lv}
        v^{i}=\frac{t^{\alpha}u^{i}}{\sqrt{1+t^{2\alpha}\delta_{k\ell}u^{k}u^{\ell}}}, \qquad L=\log\left(\left(\frac{t}{t_0}\right)^{3\alpha(1+K)}\frac{\rho}{\rho_0}\right).
    \end{equation}
    Note that, in $(L,v)$-coordinates the solution in \eqref{eq:background} becomes trivial, i.e.,
    \begin{equation*}
        (\bar{L},\bar{v})(t)\coloneqq \left(L(t,\bar{\rho}(t)),v(t,\bar{u}(t)))\right)=(0,0).
    \end{equation*}
    \subsection{The equations of motion in symmetry}
    
    Using the fluid coordinates $U=(L,v)$ introduced in \eqref{eq:Lv}, we can write the plane-symmetric relativistic Euler equations as the quasilinear first-order system
	\begin{equation}\label{system}
		U_{t}(t,x)+t^{-\alpha}A(U(t,x))U_{x}(t,x)=t^{-1}H(U(t,x),t,x),
	\end{equation}
	where $ A $ and $ H $ are given by
	\begin{equation}\label{AH}
		A=\begin{pmatrix}
			\frac{1-K}{1-Kv^{2}}v & \frac{1+K}{1-Kv^{2}} \\ \frac{K}{1+K}\frac{(1-v^{2})^{2}}{1-Kv^{2}} & \frac{1-K}{1-Kv^{2}}v
		\end{pmatrix}, \quad H=\begin{pmatrix}
			\frac{\alpha(1+K)(1-3K)}{1-Kv^{2}}v^{2} \\ -\alpha(1-3K)v+\frac{\alpha (1-K)(1-3K)}{1-Kv^{2}}v^{3}
	\end{pmatrix}.
	\end{equation}
	While tedious to do by hand, deriving \eqref{AH} is a straightforward application of the chain rule. For additional discussions of this, we refer to, e.g., \cite{fajman2025arxiv,fajman2025cqg,speck2013}. 

    The following two propositions are slight variations of textbook results. For proofs of these results, we refer to \cite{serre2007}.
    \begin{prop}[Local well-posedness]\label{localwellposed}
        		Let $s >\frac{1}{2}+1$, $t_{0}>0$, and $\delta>0$. Then, there exists a $T_{\delta}>t_{0}$, such that for each $\mathring{L}, \mathring{v} \in H^{s}(S^{1})$ satisfying
		\begin{align*}
			\|\mathring{L}\|_{H^{s}}+\|\mathring{v}\|_{H^{s}}\leq \delta,
		\end{align*}
		there exist unique functions $(L,v):[t_{0},T_{\delta})\to {H}^{s}(S^{1})$, such that
		\begin{align*}
			L&\in C^{0}([t_{0},T_{\delta}),H^{s}(S^{1}))\cap C^{1}([t_{0},T_{\delta}),H^{s-1}(S^{1})),\\
			v&\in C^{0}([t_{0},T_{\delta}),H^{s}(S^{1}))\cap C^{1}([t_{0},T_{\delta}),H^{s-1}(S^{1})),
		\end{align*}
		with 
		\begin{equation*}
			L(t_{0})=\mathring{L}, \qquad v(t_{0})=\mathring{v},
		\end{equation*}
		and so that $(L(\cdot),v(\cdot))$ solve \eqref{system} on $[t_{0},T_{\delta})\times S^{1}$. Furthermore, the solution map that maps $(\mathring{L}, \mathring{v}) \mapsto (L,v)$ is continuous. 
    \end{prop}

    \begin{prop}[Continuation criterion]\label{continuation}
        Let $(L,v)\in (C^{1}([t_0,t_*); S^{1}))^{2}$ be the classical solution of \eqref{system} and suppose that this solution cannot be extended classically to $t_*$. Then 
        \begin{equation*}
            \limsup_{t\to t_*}\|(L,v)\|_{W^{1,\infty}}=\infty.
        \end{equation*}
    \end{prop}

    \begin{remark}[Lifespan]
        Propositions \ref{localwellposed} and \ref{continuation} allow us to extend a solution $S$ uniquely up to a maximal, potentially infinity time $T_\text{max}$. We refer to $T_{\text{max}}$ as the \emph{lifespan} of $S$. 
    \end{remark}
    
    \begin{remark}[Smallness assumption]
        For the rest of this manuscript, we assume $|v|$ to be sufficiently small so that all function of $v$ that make up $A$ and $H$ are smooth. Proposition \ref{localwellposed} with sufficiently small data guarantees that an interval $[t_0,t_1)$ satisfying this assumption always exists. In fact, Lemma \ref{velbounded} ensures that this condition is then satisfied during the whole classical lifespan of the solution. 
    \end{remark}
    
	\subsection{Riemann invariants} 
    
    In this section, we will derive the Riemann invariants of the system in \eqref{system}. This allows us to write the equations of motion in a diagonal form 
    \begin{equation*}
        \partial_t \begin{pmatrix}
            R_{+}\\ R_{-}
        \end{pmatrix}+t^{-\alpha}\begin{pmatrix}
            \lambda_{+}(R_+,R_{-}) & 0\\
            0 & \lambda_{-}(R_+,R_{-})
        \end{pmatrix}\partial_x\begin{pmatrix}
            R_{+}\\ R_{-}
        \end{pmatrix}=t^{-1}\begin{pmatrix}
            S_+(R_+,R_{-})\\S_-(R_+,R_{-})
        \end{pmatrix}.
    \end{equation*}
    For details about Riemann invariants, we refer to \cite{hoermander1997,alinhac1995,evans1998}.

    \begin{lem}\label{Riemann}
        Suppose that $(L,v)$ is a sufficiently small solution of \eqref{AH} and define $R_\pm$ via
        \begin{equation}\label{eq:riemann}
		      R_{\pm}=\frac{\sqrt{K}}{1+K}L \pm \frac{1}{2}\log\frac{1+v}{1-v}
	    \end{equation}
        Then the map $(L,v)\mapsto (R_+(L,v),R_-(L,v))$ is a diffeomorphism and
        \begin{equation}\label{eq:chareom}
            D_{\pm}R_{\pm}=t^{-1}S_{\pm}(R_+,R_-),
        \end{equation}
        where 
        \begin{align*}
            D_\pm &=(\partial_{t}+t^{-\alpha}\lambda_{\pm}(R_+,R_-)\partial_{x}),\\
            \lambda_{\pm}(R_+,R_-)&=\frac{v(R_+,R_-)\pm\sqrt{K}}{1\pm v(R_+,R_-)\sqrt{K}},\\
            S_\pm(R_+,R_-)&=\mp \alpha (1-3K)\frac{v(R_+,R_-)}{1\pm\sqrt{K}v(R_+,R_-)}.
        \end{align*}
    \end{lem}

    \begin{remark}
        From this point on, for the sake of notational clarity, we will suppress the explicit dependence on the Riemann invariants. Note, however that, e.g., $v$ is considered as $v:\mathbb{R}^{2}\to \mathbb{R}$ and hence, 
        \begin{equation*}
            \partial_x v=(\partial_1v)(R_+,R_-)\partial_xR_++(\partial_2 v)(R_+,R_-)\partial_xR_-.
        \end{equation*}
        We will always consider the slots of $L$ and $v$ as functions of $R_-$ and $R_+$ and vice versa as suggested by the notation $(L,v)\mapsto (R_+(L,v),R_-(L,v))$, e.g., 
        \begin{equation*}
            \partial_v R_-(L,v)=(\partial_2 R_-)(L,v), \qquad \partial_{R_+} L(R_+,R_-)=(\partial_1 L)(R_+,R_-).
        \end{equation*}
    \end{remark}
    
    \begin{proof}[Proof of Lemma \ref{Riemann}]
	Straightforward calculation yields that the eigenvalues and left eigenvectors of $ A $ in \eqref{AH} are given by 
	\begin{equation}\label{eigen}
		\lambda_{\pm}=\frac{
		v\pm\sqrt{K}}{1\pm v\sqrt{K}}, \qquad l_{\pm}=\begin{pmatrix}
			\frac{\sqrt{K}}{1+K}\\ \pm \frac{1}{(1-v^{2})}
		\end{pmatrix}. 
	\end{equation}
	Therefore, the corresponding Riemann invariants $ R_{\pm} $ are then given by
	\begin{equation*}
		R_{\pm}(L,v)=\frac{\sqrt{K}}{1+K}L \pm \frac{1}{2}\log\frac{1+v}{1-v}. 
	\end{equation*}
	By construction, they satisfy
    \begin{equation*}
        \lambda_\pm \partial_x R_\pm = \lambda_\pm (\partial_1 R_\pm,\partial_2 R_\pm)\begin{pmatrix}
            \partial_x L \\ \partial_x v
        \end{pmatrix}=\lambda_\pm l_\pm^{T}\begin{pmatrix}
            \partial_x L \\ \partial_x v
        \end{pmatrix}=l_\pm^{T}A\begin{pmatrix}
            \partial_x L \\ \partial_x v
        \end{pmatrix}.
    \end{equation*}
    Hence, 
	\begin{align*}
		D_{\pm}R_{\pm}&=(\partial_{t}+t^{-\alpha}\lambda_{\pm}\partial_{x})R_{\pm}\\
		&=\partial_{1}R_{\pm}(\partial_{t}+t^{-\alpha}\lambda_{\pm}\partial_{x})L+\partial_{2}R_{\pm}(\partial_{t}+t^{-\alpha}\lambda_{\pm}\partial_{x})v\\
		&=l_{\pm}^{T}\partial_{t}\begin{pmatrix}
			L \\ v
		\end{pmatrix}+l_{\pm}^{T}t^{-\alpha}\lambda_{\pm}\partial_{x}\begin{pmatrix}
		L \\ v
	\end{pmatrix}=t^{-1}l_{\pm}^{T}H. 
	\end{align*}
    The source terms $l_{\pm}^{T}H$ can be written as
        \begin{equation*}
            S_\pm=l_{\pm}^{T}H=\mp \alpha (1-3K)\frac{v}{1\pm\sqrt{K}v}.
        \end{equation*}
    \end{proof}

    \begin{remark}
        Note that, due to the presence of a source in \eqref{eq:chareom}, the Riemann invariants are not preserved along characteristics. However, with knowledge of the solution on a time slab, and considering $R_-$ as fixed we can still analyze $R_+$ as solution to an ODE. 
    \end{remark}
    
	\section{Derivation of a Riccati-type equation}\label{sec:riccati}
	In this section, we will derive an ODE for $\partial_x R_+$ along a characteristic of $R_+$. The choice of $ R_{+} $ over $R_-$ is arbitrary, due to the symmetry of the equations of motion. 
    
    Commuting \eqref{eq:chareom} with $ \partial_{x} $, we find 
	\begin{align*}
		D_{+}\partial_{x}R_{+}&=t^{-1}\partial_{x}S_{+}-t^{-\alpha}\partial_{x}\lambda_{+} \partial_{x}R_{+}\\
		&=t^{-1}\Big(\partial_{1}S_{+}\partial_{x}R_{+}+\partial_{2}S_{+}\partial_{x}R_{-}\Big)-t^{-\alpha}\Big(\partial_{1}\lambda_{+}\partial_{x}R_{+}+\partial_{2}\lambda_{+} \partial_{x}R_{-}\Big)\partial_{x}R_{+}\\
		&=-t^{-\alpha}\partial_{1}\lambda_{+}(\partial_{x}R_{+})^{2}+t^{-1}\partial_{1}S_{+}\partial_{x}R_{+}+t^{-1}\partial_{2}S_{+}\partial_{x}R_{-}-t^{-\alpha}\partial_{2}\lambda_{+} \partial_{x}R_{+}\partial_{x}R_{-}. 
	\end{align*}
	We eliminate the mixed term which is the last term of this equation by the following consideration. 
    \begin{defin}
    We define the function $\tau:\mathbb{R}^{2}\to\mathbb{R}$ as the solution to the following ODE
	\begin{equation}\label{eq:tau}
    \begin{aligned}
		(\partial_{2}\log\tau)(\cdot,\cdot) &= \frac{(\partial_{2}\lambda_{+})(\cdot,\cdot)}{\lambda_{+}(\cdot,\cdot)-\lambda_{-}(\cdot,\cdot)},\\
        \tau(\cdot, 0)&=\tau_0,
    \end{aligned}
	\end{equation}
    where $\tau_0\in\mathbb{R}$. 
    \end{defin}
    \begin{remark}[Initializing $\tau_0$]\label{rem:tau0}
        From this point on, for the sake of convenience, we choose $\tau_0=1$. 
    \end{remark}
	\begin{remark}
	    Note that the denominator in \eqref{eq:tau} is bounded away from zero under the smallness assumption. This crucial feature is often referred to as \emph{strict hyperbolicity}, meaning that wave speeds never coincide. 
	\end{remark}
    With this definition at hand, we find that
	\begin{align*}
		D_{+}\tau &= \partial_{1}\tau D_{+}R_{+}+\partial_{2}\tau D_{+}R_{-}\\
		&=\partial_{1}\tau t^{-1} S_{+}+\tau \frac{\partial_{2}\lambda_{+}}{\lambda_{+}-\lambda_{-}}D_{+}R_{-}\\
		&=\partial_{1}\tau t^{-1} S_{+}+\tau \frac{\partial_{2}\lambda_{+}}{\lambda_{+}-\lambda_{-}}\Big(t^{-1}S_{-}+t^{-\alpha}(\lambda_{+}-\lambda_{-})\partial_{x}R_{-}\Big)\\
		&=\partial_{1}\tau t^{-1} S_{+}+\partial_2\tau t^{-1}S_{-}+\tau t^{-\alpha}\partial_{2}\lambda_{+}\partial_{x}R_{-},
	\end{align*}
	and therefore
	\begin{align*}
		D_{+}(\tau\partial_{x}R_{+})&=\tau\Big(-t^{-\alpha}\partial_{1}\lambda_{+}(\partial_{x}R_{+})^{2}+t^{-1}\partial_{1}S_{+}\partial_{x}R_{+}+t^{-1}\partial_{2}S_{+}\partial_{x}R_{-}\\
		&\quad-t^{-\alpha}\partial_{2}\lambda_{+} \partial_{x}R_{+}\partial_{x}R_{-}\Big)+\partial_{x}R_{+}\Big(\partial_{1}\tau t^{-1} S_{+}+\partial_2\tau t^{-1}S_{-}+\tau t^{-\alpha}\partial_{2}\lambda_{+}\partial_{x}R_{-}\Big)\\
		&=-t^{-\alpha}\tau^{-1}\partial_{1}\lambda_{+} (\tau\partial_{x}R_{+})^{2}+t^{-1}\Big(\partial_{1}S_{+}+\tau^{-1}\partial_{1}\tau S_{+}+\tau^{-1}\partial_2\tau S_{-}\Big)(\tau\partial_{x}R_{+})\\
		&\quad+\tau t^{-1}\partial_{2}S_{+}\partial_{x}R_{-}
	\end{align*}
	We can now additionally control involving $ \partial_{x}R_{-} $ the last term by introducing the following mechanism.
    \begin{defin}
        We define the function $\xi:\mathbb{R}^{3}\to \mathbb{R}$ as the unique solution to the ODE
        \begin{equation}\label{eq:xidef}
        \begin{aligned}
		\partial_{2}\xi(\cdot,\cdot,t)&=t^{-1+\alpha}\frac{\tau(\cdot,\cdot)}{\lambda_{+}(\cdot,\cdot)-\lambda_{-}(\cdot,\cdot)}\partial_{2}S_{+}(\cdot,\cdot),\\
        \xi(\cdot,0,t)&=0.
        \end{aligned}
	\end{equation}
    \end{defin}

    \begin{remark}[Explicit form of $\xi$]
    In fact, we can calculate $\xi$ explicitly in terms of $t$ and $\tau$. Simple calculations show that
    \begin{equation*} 
    \partial_2 S_+= -\frac{(1-3 K)\alpha}{(1-K)}\partial_2 \lambda_+,
    \end{equation*}
    and so the definition of $\xi$ and $\tau$ imply that 
    \begin{equation*}
    \partial_2 \xi = -\frac{(1-3 K)\alpha}{(1-K)}t^{\alpha-1}\partial_2\tau.
    \end{equation*}
    Integrating this and applying the initial conditions satisfied by $\tau$ and $\xi$ then shows that
    \begin{equation*}
    \xi = \frac{(1-3 K)\alpha}{(1-K)} t^{\alpha-1}(1-\tau).
    \end{equation*}
    \end{remark}

	We have  that
	\begin{align*}
		D_{+}(\tau\partial_{x}R_{+}-\xi)&=D_{+}(\tau\partial_{x}R_{+})-\partial_{1}\xi D_{+}R_{+}-\partial_{2}\xi D_{+}R_{-}-\partial_t \xi\\
		&=D_{+}(\tau\partial_{x}R_{+})-\partial_{1}\xi t^{-1}S_{+}-{t^{-1}(\alpha-1)\xi} \\
        &\quad -t^{-1+\alpha}\frac{\tau}{\lambda_{+}-\lambda_{-}}\partial_{2}S_{+}\Big(S_{-}t^{-1}+t^{-\alpha}(\lambda_{+}-\lambda_{-})\partial_{x}R_{-}\Big).
	\end{align*}
	Hence, finally we find that along the $ R_{+} $-characteristic, we have
	\begin{equation*}
		D_{+}(\tau\partial_{x}R_{+}-\xi)=A_{2}(\tau\partial_{x}R_{+}-\xi)^{2}+A_{1}(\tau\partial_{x}R_{+}-\xi)+A_{0},
	\end{equation*}
	with 
	\begin{align*}
		A_{2}&=-t^{-\alpha}\tau^{-1}\partial_{1}\lambda_{+},\\
		A_{1}&=t^{-1}\Big(\partial_{1}S_{+}+\tau^{-1}\partial_{1}\tau S_{+}+\tau^{-1}\partial_2\tau S_{-}\Big){+2A_2\xi},\\
		A_{0}&=-\partial_{1}\xi t^{-1}S_{+}-t^{-1+\alpha}\frac{\tau}{\lambda_{+}-\lambda_{-}}\partial_{2}S_{+}S_{-}t^{-1}\\
		&\quad +\xi t^{-1}\Big(\partial_{1}S_{+}+\tau^{-1}\partial_{1}\tau S_{+}+\tau^{-1}\partial_2\tau S_{-}\Big){+A_2 \xi^2-t^{-1}(\alpha-1)\xi}.
	\end{align*}
	Equivalently, with $ \zeta=\xi-\tau\partial_{x}R_{+} $ we have 
	\begin{equation}\label{riccati}
		D_{+}\zeta=\tilde{A}_{2}\zeta^{2}+\tilde{A}_{1}\zeta +\tilde{A}_{0}. 
	\end{equation}
    with
    \begin{equation}\label{eq:deftildeA}
        \tilde{A}_2=-A_2, \qquad \tilde{A}_1=A_1, \qquad \tilde{A}_0=-A_0.
    \end{equation}
    \begin{remark}
        The reason for the sign change in \eqref{riccati} is so that the leading order coefficient, $\tilde{A}_2$, drives blowup towards $+\infty$ rather than $-\infty$.
    \end{remark}
	\section{Preliminary estimates}\label{sec:est}

    The result in this article hinges on the fact that the velocity is bounded by the initial data, as long as a classical solution exists. The following section makes this rigorous. 
    
    \subsection{Invariant region argument}
    We start with a definition.
	\begin{defin}
		Let $\mu,\delta>0$ and $\rho \in C^{1}([t_{0},\infty);\mathbb{R})$ with $\rho(t_{0})=0$ and $\rho^{\prime}(t)>0$ for all $t\in [t_{0},\infty)$. We define 
        \begin{equation}
            \begin{aligned}
                D_{\mu}&\coloneqq \{(x,y)\in \mathbb{R}^{2}\vert \max(|x|,|y|)< \mu\},\\
                D_{\mu}^{\delta}(t)&\coloneqq \{(x,y)\in \mathbb{R}^{2}\vert \max(|x|,|y|)< \mu +\delta \rho(t) \}.
            \end{aligned}
        \end{equation}
	\end{defin}
	\begin{lem}\label{velbounded}
		Let $t_1>t_0$, $\mu>0$ and $(R_+,R_-)$ be a classical solution on $[t_0,t_1]$ with $$(R_{+},R_-)(t_0,x)\in D_\mu$$ for all $x\in S^1$. Then $(R_{+},R_-)$ remains in $\overline{D_\mu}$ for all $t\in [t_0,t_1]$ and $x\in S^1$.
	\end{lem}
	\begin{proof}
		Let $\delta>0$. We start by proving the weaker statement that $(R_{+},R_-)$ never leaves the expanding domain $D_\mu^\delta$. 
        
        Assume that the contrary is true. Then, there exists a minimal time $T$ with $t_0<T\leq t_1$ and $X\in S^1$, such that
        \begin{equation*}
            \max(|R_{+}(T,X)|;|R_-(T,X)|)=\mu +\delta \rho(T).
        \end{equation*}
        This implies that one of the following most be true:
        \begin{align*}
            R_{+}(T,X)&=\mu +\delta \rho(T), &R_{+}(T,X)=-(\mu +\delta \rho(T)),\\
            R_-(T,X)&=\mu +\delta \rho(T), &R_-(T,X)=-(\mu +\delta \rho(T)).
        \end{align*}
        Assume that the first equation is satisfied and let $\gamma$ be the characteristic $R_{+}$, such that $\gamma(T)=X$. Hence, 
        \begin{equation*}
            R_{+}(T,\gamma(T))=\mu +\delta \rho(T).
        \end{equation*}
        Furthermore, we define the function $q:[t_0,T]\to\mathbb{R}$ via
        \begin{equation*}
            q(t)\coloneqq R_{+}(t,\gamma(t))-(\mu +\delta \rho(t)).
        \end{equation*}
        Now, via the assumption that $T$ is minimal, i.e., it is the time of first exit, we infer that $q(t)<0$ for all $\tilde{t}<t<T$ and $q(T)=0$, provided $T-\tilde{t}>0$ is sufficiently small. As $q$ is differentiable, this implies that $q^\prime(T)\geq 0$. 

        Now, differentiating $q$ yields
        \begin{equation*}
            q^\prime(t)= (D_{+}R_{+})(t,\gamma(t))-\delta \rho^\prime(t)=-t^{-1}\alpha (1-3K)\frac{v}{1+\sqrt{K}v}(t,\gamma(t))-\delta \rho^\prime(t).
        \end{equation*}
        Inverting \eqref{eq:riemann} yields
        \begin{equation}\label{eq:vinR}
            v=\frac{e^{R_{+}}-e^{R_-}}{e^{R_{+}}+e^{R_-}},
        \end{equation}
        but since $R_{+}(T,X)$ is maximal, we have that $R_{+}(T,X)\geq R_-(T,X)$ and hence,
        \begin{equation*}
           {v(T,X)\geq 0}.
        \end{equation*}
        This implies that
        \begin{equation*}
            q^\prime(T)\leq -\delta \rho^\prime(T)<0,
        \end{equation*}
        which is a contradiction. 

        Now assume that $R_{+}(T,X)=-(\mu +\delta \rho(T))$. Again, as $t$ approaches $T$, $q$, now defined via
        \begin{equation*}
            q(t)=R_{+}(t,\gamma(t))+\mu+\delta \rho(t),
        \end{equation*}
        is positive as well as $q(T)=0$. Hence, $q^\prime(T)\leq 0$. However, since $R_{+}$ is now minimal, $v(T,X)\leq 0$ and hence, 
        \begin{equation*}
            q^\prime(T)\geq \delta \rho^\prime(T)>0.
        \end{equation*}
        The analogous statements for $R_-$ can be similarly derived, acknowledging the change of sign in the source. Alternatively, one can note the symmetry of the problem in the Riemann invariants. 

        We have therefore shown that the solution can never leave the expanding domain $D_\mu^\delta$. However, since this proof was independent of the size of $\delta$, we can let $\delta \to 0$ and conclude that $(R_{+},R_-)$ can never leave $\overline{D_\mu}$. 
	\end{proof}

    \subsection{Estimates on the Riccati-coefficients}

    \begin{lem}\label{basicest}
        Given the estimate $\max(|R_{+}|,|R_-|)<\epsilon$ where $\epsilon>0$ is sufficiently small, we have the following estimates on composite quantities:
        \begin{alignat*}{2}
            e^{R_{+}}&=1+\mathcal{O}(\epsilon),&\qquad 
            e^{R_-}&=1+\mathcal{O}(\epsilon),\\
            v&=\mathcal{O}(\epsilon),&S_\pm&=\mathcal{O}(\epsilon),\\
            \lambda_+&=\sqrt{K}+\mathcal{O}(\epsilon), &\lambda_-&=-\sqrt{K}+\mathcal{O}(\epsilon),\\
            \partial_1 v&=\frac{1}{2}+\mathcal{O}(\epsilon^2),& \partial_{2} v&=-\frac{1}{2}+\mathcal{O}(\epsilon^2),\\
            \partial_1 \lambda_+ &= \frac{1-K}{2}+\mathcal{O}(\epsilon), & \partial_2\lambda_+ &= -\frac{1-K}{2}+\mathcal{O}(\epsilon),\\
            \partial_1\tau&=\mathcal{O}(\epsilon) ,& \partial_2\tau&= -\frac{1-K}{4\sqrt{K}}+\mathcal{O}(\epsilon),\\
            \partial_{1}S_+&=-\frac{\alpha}{2}(1-3K)+\mathcal{O}(\epsilon),& \partial_{2}S_+&=\frac{\alpha}{2}(1-3K)+\mathcal{O}(\epsilon),\\
            \tau &= 1+\mathcal{O}(\epsilon),& \xi&=t^{-1+\alpha}\mathcal{O}(\epsilon).
        \end{alignat*}
    \end{lem}

    \begin{proof}
        Using Taylor expansion, it is straightforward to see that
        \begin{equation*}
            1-e^{R_{+}}=R_{+}+\mathcal{O}(R_{+}^{2})=\mathcal{O}(\epsilon).
        \end{equation*}
        Clearly, the same is true for $e^{R_-}$. Using \eqref{eq:vinR}, we find that
        \begin{equation*}
            |v|=\frac{|e^{R_{+}}-e^{R_{-}}|}{e^{R_{+}}+e^{R_{-}}}\leq |e^{R_{+}}-e^{R_{-}}|\leq|1-e^{R_{+}}|+|1-e^{R_-}| =\mathcal{O}(\epsilon). 
        \end{equation*}
        The estimate for $\lambda_\pm, \partial_{1,2}v, \partial_1\lambda_+$ are obtained in a quite similar fashion. Now, keeping in mind \eqref{eq:tau}, we see that
        \begin{equation}\label{eq:tauint}
            \log\tau(R_+,R_-)-\log\tau_0=\int_{0}^{R_-}\frac{\partial_2\lambda_{+}(R_+,s)}{\lambda_{+}(R_+,s)-\lambda_{-}(R_+,s)}ds.
        \end{equation}
        Taking into account the previous considerations, this yields that there exists a constant $C>0$ such that
        \begin{equation*}
            |\log\tau(R_+,R_-)-\log\tau_0| \leq |R_-|(C+\mathcal{O}(\epsilon)).
        \end{equation*}
        By remark \ref{rem:tau0}, this yields
        \begin{equation*}
            \tau(R_+,R_-)=e^{\mathcal{O}(\epsilon)}=1+\mathcal{O}(\epsilon). 
        \end{equation*}
        Taking the derivative with respect to $R_+$ in \eqref{eq:tauint} produces 
        \begin{equation*}
            \partial_{R_+}\tau(R_+,R_-)=\tau \int_{0}^{R_-}\partial_{R_+}\left(\frac{\partial_2\lambda_{+}(R_+,s)}{\lambda_{+}(R_+,s)-\lambda_{-}(R_+,s)}\right)ds.
        \end{equation*}
        which is clearly bounded by $\mathcal{O}(\epsilon)$. On the other hand, using the fundamental theorem of calculus, we find that
        \begin{equation*}
            \partial_{R_-}\tau(R_+,R_-)=\tau \frac{\partial_2\lambda_{+}(R_+,R_-)}{\lambda_{+}(R_+,R_-)-\lambda_{-}(R_+,R_-)},
        \end{equation*}
        which asymptotes to $-\frac{1-K}{4\sqrt{K}}$. Similar considerations produce the other estimates in a straightforward manner. 
    \end{proof}
    
	In the proof of Theorem \ref{blowup}, we will be interested in the leading order behavior of the coefficients $ \tilde{A}_{0},\tilde{A}_1$ and $\tilde{A}_2$ in the regime of small Riemann invariants $R_\pm$. Hence, we compute the following. 
	\begin{lem}\label{lem:riccatidiff}
		Given the estimate $\max(|R_{+}|,|R_-|)<\epsilon$ where $\epsilon>0$ is sufficiently small, we establish the following difference growth rates.
		\begin{align}
			\left|\frac{1-K}{2}t^{-\alpha}-\tilde{A}_{2}\right|&\leq t^{-\alpha} \mathcal{O}(\epsilon),\label{eq:asymptA2}\\
			\left|-\frac{\alpha}{2}(1-3K)t^{-1}-\tilde{A}_{1}\right|&\leq t^{-1} \mathcal{O}(\epsilon),\\
			|\tilde{A}_{0}|&\leq t^{-2+\alpha} \mathcal{O}(\epsilon).
		\end{align}
	\end{lem}

    \begin{proof}
        The proof will be a rudimentary application of Lemma \ref{basicest} to the expressions in \eqref{eq:deftildeA}. The estimate for \eqref{eq:asymptA2} is immediate. A closer look at $\tilde{A}_1$ reveals that most of the terms therein are perturbative, as
        \begin{align*}
            |\tau^{-1}\partial_1 \tau S_{+}|&\leq|\tau^{-1}| |\partial_1 \tau| |S_{+}|\lesssim \mathcal{O}(\epsilon),\\
            |\tau^{-1}\partial_2 \tau S_-|&\leq |\tau^{-1}| |\partial_2\tau||S_-|\lesssim \mathcal{O}(\epsilon),\\
            |A_2\xi|&\leq t^{-\alpha}|\tau^{-1}||\partial_1\lambda_+||\xi|\lesssim t^{-1}\mathcal{O}(\epsilon).
        \end{align*}
         The leading order behavior is thus determined by $t^{-1}\partial_1 S_+$ and given by $B_1$. Similar arguments can be applied $\tilde{A}_0$, where we note in particular that
        \begin{equation*}
            |A_2\xi^{2}|\lesssim t^{-\alpha}t^{-2+2\alpha}\mathcal{O}(\epsilon)=t^{-2+\alpha}\mathcal{O}(\epsilon).
        \end{equation*}
        This is the desired decay rate as indicated above. 
    \end{proof}

    \section{Proof of Theorem \ref{blowup}} \label{sec:thm1}

    We are now ready to prove Theorem \ref{blowup}. It combines several aspects of the equations of motion of the characteristics \eqref{eq:chareom}. First, \eqref{riccati} is a Riccati-type equation that determines the evolution of the quantity $\zeta$ along a specific characteristic of $R_+$. The coefficients in \eqref{riccati} are controlled by the invariant region argument, Lemma \ref{velbounded}, and the consequential uniform estimates in Lemma \ref{basicest} and \ref{lem:riccatidiff}. In addition, these Lemmata provide control over $\xi$ and $\tau$, guaranteeing that blowup of $\zeta$ corresponds to blowup of $\partial_xR_+$. Finally, we show that the we can construct a specific sequence of data that displays this blowup behavior while going to $0$ in the desired topology. 
    
    \begin{proof}[Proof of Theorem \ref{blowup}] 
    Let $\epsilon>0$. We choose initial data $(\mathring{R}_+,\mathring{R}_-)\in C^{\infty}(S^1)\times C^{\infty}(S^1)$, such that
        \begin{equation}\label{eq:data}
            \mathring{R}_-=0, \qquad \|\mathring{R}_+\|_{L^{\infty}}<\epsilon, \qquad \|\mathring{R}_{+}^{\prime}\|_{L^{\infty}}=\delta(\epsilon),
        \end{equation}
        where $\delta(\epsilon)>0$ is a function to be chosen later on. Furthermore, we assume that there exists a $z\in S^{1}$ such that $\mathring{R}^{\prime}_+(z)=-\delta(\epsilon)$. 
        
        Let $(R_+,R_-)$ be the unique classical solution to \eqref{system} with initial data $(\mathring{R}_+,\mathring{R}_-)$ on the time interval $[t_0,t_*)$. This is guaranteed to exist by Proposition \ref{localwellposed}, provided $t_*$ is sufficiently small. We define the characteristic $\varrho_+$ via
        \begin{equation*}
            \frac{d}{dt}\varrho_+(t)= t^{-\alpha}\lambda_+(t,\varrho_+(t)),\qquad \varrho_+(t_0)=z.
        \end{equation*}
        By \eqref{eq:tau} and \eqref{eq:xidef} this implies that 
        \begin{equation*}
            \zeta(t_0,\varrho_+(t_0))=\xi(\mathring{R}_+(z),0,t_0)-\tau(\mathring{R}_+(z),0)\mathring{R}_+^{\prime}(z)=-\mathring{R}_+^{\prime}(z)=\delta(\epsilon).
        \end{equation*}
        The evolution of $\zeta(\cdot,\varrho_+(\cdot))$ is determined by \eqref{riccati}. For the sake of simplicity, we introduce the constant
        \begin{equation*}
            \gamma\coloneqq \alpha\frac{3}{2}(1-K)-1.
        \end{equation*}
        Note that the statement of Theorem \ref{blowup} includes the two cases $\gamma=0$ and $\gamma<0$, which we will analyze separately. Throughout the proof, we will use explicit constants $c_i$ where $i\in \mathbb{N}$. These constants will always be tacitly assumed to be greater than $0$ and sufficiently small or large depending on the context. 

        \textbf{The subcritical case ($\gamma<0$).} We start with the subcritical case, i.e., $\gamma<0$. Consider the evolution of $\zeta$ along $\varrho_+$. Using the notation of Lemma \ref{lem:blowup}, the uniform bound of $(R_+,R_-)$ implied by Lemma \ref{velbounded}, and the estimate in Lemma \ref{lem:riccatidiff}, we have the following estimate
        \begin{equation}\label{eq:K1}
        \begin{aligned}
            K_1(t^{\prime})&=\int_{t_0}^{t^{\prime}}\tilde{A}_1(t,\varrho_+(t))dt=-\int_{t_0}^{t^{\prime}}\left(\frac{\alpha}{2}(1-3K)+\mathcal{O}(\epsilon)\right)t^{-1}dt\\
            &=-\left(\frac{\alpha}{2}(1-3K)+\mathcal{O}(\epsilon)\right)\log\frac{t^{\prime}}{t_0}.
            \end{aligned}
        \end{equation}
        Provided $\epsilon$ is sufficiently small such that for all instances of $\mathcal{O}(\epsilon)$ we have $\gamma+\mathcal{O}(\epsilon)<0$, we can compute
        \begin{align*}
            K_0(t_*)&=\int_{t_0}^{t_*}e^{-K_1(t)}|\tilde{A}_0(t,\varrho_+(t))|dt=\int_{t_0}^{t_*}\left(\frac{t}{t_0}\right)^{\frac{\alpha}{2}(1-3K)+\mathcal{O}(\epsilon)}\mathcal{O}(\epsilon)t^{-2+\alpha}dt\\
            &=\mathcal{O}(\epsilon)t_0^{-\frac{\alpha}{2}(1-3K)+\mathcal{O}(\epsilon)}\frac{t_*^{\gamma+\mathcal{O}(\epsilon)}-t_0^{\gamma+\mathcal{O}(\epsilon)}}{\gamma+\mathcal{O}(\epsilon)},
        \end{align*}
        and
        \begin{align*}
            K_2(t_*)&=\int_{t_0}^{t_*}e^{K_1(t)}\tilde{A}_2(t,\varrho_+(t))dt=\left(\frac{1-K}{2}+\mathcal{O}(\epsilon)\right)\int_{t_0}^{t_*}\left(\frac{t}{t_0}\right)^{-\frac{\alpha}{2}(1-3K)+\mathcal{O}(\epsilon)}t^{-\alpha}dt\\
            &=\left(\frac{1-K}{2}+\mathcal{O}(\epsilon)\right)t_0^{\frac{\alpha}{2}(1-3K)+\mathcal{O}(\epsilon)}\frac{t_*^{-\gamma+\mathcal{O}(\epsilon)}-t_0^{-\gamma+\mathcal{O}(\epsilon)}}{-\gamma+\mathcal{O}(\epsilon)}.
        \end{align*}
        We calculate
        \begin{equation}\label{eq:blowupcon}
        \begin{aligned}
            K_2(t_*)^{-1}+K_{0}(t_*)&= \left(\left(\frac{1-K}{2}+\mathcal{O}(\epsilon)\right)t_0^{\frac{\alpha}{2}(1-3K)+\mathcal{O}(\epsilon)}\frac{t_*^{-\gamma+\mathcal{O}(\epsilon)}-t_0^{-\gamma+\mathcal{O}(\epsilon)}}{-\gamma+\mathcal{O}(\epsilon)}\right)^{-1}\\
            &\quad+\mathcal{O}(\epsilon)t_0^{-\frac{\alpha}{2}(1-3K)+\mathcal{O}(\epsilon)}\frac{t_*^{\gamma+\mathcal{O}(\epsilon)}-t_0^{\gamma+\mathcal{O}(\epsilon)}}{\gamma+\mathcal{O}(\epsilon)}\\
            &\simeq \left(\left(\frac{t_*}{t_0}\right)^{-\gamma+\mathcal{O}(\epsilon)}-1\right)^{-1}-\mathcal{O}(\epsilon)\left(\left(\frac{t_*}{t_0}\right)^{\gamma+\mathcal{O}(\epsilon)}-1\right)\\
            &\simeq \left(\left(\frac{t_*}{t_0}\right)^{-\gamma+\mathcal{O}(\epsilon)}-1\right)^{-1}+\mathcal{O}(\epsilon),
        \end{aligned}
        \end{equation}
        where we used the fact that for all $ t\in [t_0,\infty)$,
        \begin{equation*}
            0\leq \left|\left(\frac{t}{t_0}\right)^{\gamma+\mathcal{O}(\epsilon)}-1\right|\leq 1.
        \end{equation*}
        A short calculation yields that \eqref{eq:cond} is satisfied if
        \begin{equation}\label{eq:subupper}
            t_{*}= t_0\left(1+\frac{1}{c_1\delta+\mathcal{O}(\epsilon)}\right)^{\frac{1}{-\gamma+\mathcal{O(\epsilon)}}}.
        \end{equation}
        Then, $\delta(\epsilon)=\bar{C}\epsilon$ with $\bar{C}>0$ sufficiently large, implies that $c_1\delta+\mathcal{O}(\epsilon) >0$. Therefore, the calculation
        \begin{equation*}
            t_*\lesssim  \left(\frac{1}{\delta+\mathcal{O}(\epsilon)}\right)^{\frac{1}{-\gamma+\mathcal{O}(\epsilon)}}\left(1+\delta+\mathcal{O}(\epsilon)\right)^{\frac{1}{-\gamma+\mathcal{O}(\epsilon)}}\simeq \delta^{\frac{1}{\gamma+\mathcal{O}(\epsilon)}}=\delta^{\frac{1}{\gamma+\mathcal{O}(\delta)}}=\delta^{\frac{1}{\gamma}}\delta^{-\frac{\gamma^{-2}\mathcal{O}(\delta)}{1+\gamma^{-1}\mathcal{O}(\delta)}}\simeq  \delta^{\frac{1}{\gamma}}
        \end{equation*}
        yields the desired upper bound on the lifespan. 
        
        Now consider the characteristic $\varrho$ with $\varrho(t_0)=x$ with $x\in S^{1}$ arbitrary. A similar calculation to \eqref{eq:blowupcon} for the ODE satisfied by $\zeta(\cdot,\varrho(\cdot))$ finds an analogous result
        \begin{equation*}
            K_2(t_*)^{-1}-K_{0}(t_*)\simeq \left(\left(\frac{t_*}{t_0}\right)^{-\gamma+\mathcal{O}(\epsilon)}-1\right)^{-1}+\mathcal{O}(\epsilon),
        \end{equation*}
        an hence, by Lemma \ref{lem:riccatiexist}, the lifespan satisfies the bound
        \begin{equation*}
            t_{*}\gtrsim t_0\left(1+\frac{1}{\delta+\mathcal{O}(\epsilon)}\right)^{\frac{1}{-\gamma+\mathcal{O(\epsilon)}}}.
        \end{equation*}
        Note that, by symmetry of $R_+$ and $R_-$, and the fact that $\mathring{R}^{\prime}_-=0$, $\partial_x R_-$ cannot blow up before $\partial_xR_+$. To see this, one can simply repeat the arguments of section \ref{sec:riccati} and \ref{sec:est}, but tracking along the characteristics of $R_-$ instead. This yields a similar equations to \eqref{riccati} for $\partial_xR_-$. This, in turn, provides a lower bound on the blowup-time of $\partial_x R_-$, by an analogous analysis to the one above. 
        
        In addition, $\xi$ and $\tau-1$ are bounded uniformly in time by Lemma \ref{basicest}, ensuring that the blowup of $\zeta$ along $\varrho_+$ is indeed in $\partial_xR_+$.
        
        \textbf{The critical case ($\gamma=0$).} Again, following $\varrho_+$, we get that
        \begin{align*}
            K_0(t_*)&=\int_{t_0}^{t_*}e^{-K_1(t)}|\tilde{A}_0(t,\varrho_+(t))|dt=\int_{t_0}^{t_*}\left(\frac{t}{t_0}\right)^{\frac{\alpha}{2}(1-3K)+\mathcal{O}(\epsilon)}\mathcal{O}(\epsilon)t^{-2+\alpha}dt\\
            &\simeq \int_{t_0}^{t_*}\left(\frac{t}{t_0}\right)^{-1+\mathcal{O}(\epsilon)}\mathcal{O}(\epsilon)dt\lesssim \mathcal{O}(\epsilon)\frac{\left(\frac{t_*}{t_0}\right)^{c_2\epsilon}-1}{c_2\epsilon}\lesssim \left(\frac{t_*}{t_0}\right)^{c_2\epsilon}-1.
        \end{align*}
        Furthermore, $K_2$ can be estimated via
        \begin{align*}
            K_2(t_*)&=\int_{t_0}^{t_*}e^{K_1(t)}\tilde{A}_2(t,\varrho_+(t))dt=\left(\frac{1-K}{2}+\mathcal{O}(\epsilon)\right)\int_{t_0}^{t_*}\left(\frac{t}{t_0}\right)^{-\frac{\alpha}{2}(1-3K)+\mathcal{O}(\epsilon)}t^{-\alpha}dt\\
            &\gtrsim \int_{t_0}^{t_*}\left(\frac{t}{t_0}\right)^{-1-c_2\epsilon}dt=\frac{1-\left(\frac{t_*}{t_0}\right)^{-c_2\epsilon}}{c_2\epsilon},
        \end{align*}
        provided $c_2$ is sufficiently large. Hence, we find that
        \begin{equation*}
            K_2(t_*)^{-1}+K_0(t_*)\lesssim \frac{c_2\epsilon}{1-\left(\frac{t_*}{t_0}\right)^{-c_2\epsilon}}+\left(\frac{t_*}{t_0}\right)^{c_2\epsilon}-1.
        \end{equation*}
        Solving the condition 
        \begin{equation*}
            c_3\delta = \frac{c_2\epsilon}{1-\left(\frac{t_*}{t_0}\right)^{-c_2\epsilon}}+\left(\frac{t_*}{t_0}\right)^{c_2\epsilon}-1,
        \end{equation*}
        we find that
        \begin{equation*}
            t_*=  t_0\left(\frac{2+c_3\delta-c_2\epsilon\pm\sqrt{(c_3\delta-c_2\epsilon)^2-4c_2\epsilon}}{2}\right)^{\frac{1}{c_2\epsilon}}.
        \end{equation*}
        We see that for the discriminant to be positive for small $\epsilon$, we need that $\delta(\epsilon) \geq \sqrt{\bar{C}^{\prime}\epsilon}$ with $\bar{C}^{\prime}>0$ sufficiently large. We pick the smaller root, as this gives us the earliest at which \eqref{eq:cond} is satisfied. Therefore, we find an upper bound given by
        \begin{equation*}
            t_{*}\lesssim \left(1+\delta\right)^{\frac{\tilde{c}_2}{\delta^{2}}}\simeq e^{\frac{\tilde{c}_2}{\delta^{2}}\log(1+\delta)}\simeq e^{\frac{\tilde{c}_2}{\delta}}. 
        \end{equation*}
        Now, again, consider the characteristic $\varrho$ with $\varrho(t_0)=x$ with $x\in S^{1}$ arbitrary. Changing the direction of the estimates, but now considering $-K_0$, we find that 
        \begin{equation*}
            -K_0(t_*)\geq c_4 \left(1-\left(\frac{t_*}{t_0}\right)^{c_2\epsilon}\right).
        \end{equation*}
        Furthermore, we have 
        \begin{equation*}
            K_2(t_*)\lesssim \frac{\left(\frac{t_*}{t_0}\right)^{c_2\epsilon}-1}{c_2\epsilon} 
        \end{equation*}
        and hence,
        \begin{equation*}
            K_2(t_*)^{-1}-K_0(t_*)\geq c_5\frac{c_2\epsilon}{\left(\frac{t_*}{t_0}\right)^{c_2\epsilon}-1} + c_4 \left(1-\left(\frac{t_*}{t_0}\right)^{c_2\epsilon}\right).
        \end{equation*}
        This leads to the bound 
        \begin{equation*}
            t_*\geq  t_0\left(1+\frac{2c_2 c_5\epsilon}{\delta+\sqrt{\delta^{2}+4c_2c_4c_5\epsilon}}\right)^{\frac{1}{c_2\epsilon}}.
        \end{equation*}
        Using the fact that $\delta =\mathcal{O}(\sqrt{\epsilon})$, via the same arguments as above, we find the upper bound
        \begin{equation*}
            t_*\gtrsim  e^{\frac{c_6}{\delta}}. 
        \end{equation*}
        
        \textbf{Construction of the initial data.} What is left to show, is that there exists a sequence of initial data satisfying \eqref{eq:data} while going to $0$ in the desired topology. To achieve that, suppose we are given a function $f\in C^{\infty}(\mathbb{R})$ with $f=0$ outside $(-\kappa,\kappa)$, where $\kappa>0$ and with the properties
        \begin{equation*}
            \|f\|_{L^{\infty}}=\| f^{\prime}\|_{L^{\infty}}=\| f^{\prime\prime}\|_{L^{\infty}}=1.
        \end{equation*}
        Furthermore, we assume that there exists $x_\text{min}$ such that $f^{\prime}(x_\text{min})=-1$, i.e., it attains its minimum  attains at $x_\text{min}$.
        
       For $\gamma<0$ and $x_0\in S^1$, we define the net $(f^{\text{sub}}_\epsilon)\subset C^{\infty}(S^{1})$ 
        \begin{equation*}
            f^{\text{sub}}_\epsilon(x)\coloneqq \epsilon f\left( \frac{x-x_0}{\beta} \right),
        \end{equation*}
        where we choose $\beta>0$ sufficiently small so that $\beta^{-1} \geq \bar{C}$, where $\bar{C}$ is the same constant as defined after \eqref{eq:subupper}. Now choosing $(\mathring{R}_+)_n=f_{\frac{1}{n}}^{\text{sub}}$ satisfies \eqref{eq:data} and, in addition,
        \begin{equation*}
            \lim_{n\to \infty} \|\mathring{R}_n\|_{H^{s}}=0
        \end{equation*}
        for all $s\geq 2$. 
        
        Now let $\gamma=0$. We define the net $(f_{\epsilon}^{\text{crit}})\subset C^{\infty}(S^{1})$ via
        \begin{equation*}
            f_{\epsilon}^{\text{crit}}(x)\coloneqq \epsilon f\left( \frac{x-x_0}{\beta \epsilon^{1-p}} \right),
        \end{equation*}
        where, again, $\beta$ is a small constant. Note that, for all $\epsilon\in (0,1)$,
        \begin{equation*}
            \|f_{\epsilon}^{\text{crit}}\|_{L^{\infty}}=\epsilon, \qquad \|(f_{\epsilon}^{\text{crit}})^{\prime}\|_{L^{\infty}}=\beta^{-1}\epsilon^{p}, \qquad \|(f_{\epsilon}^{\text{crit}})^{\prime\prime}\|_{L^{\infty}}=\beta^{-2}\epsilon^{-1+2p}.
        \end{equation*}
        Furthermore, we note that
        \begin{align*}
            \|f_{\epsilon}^{\text{crit}}\|_{L^{2}(S^{1})}&=\beta^{\frac{1}{2}}\epsilon^{1+\frac{1-p}{2}} \|f\|_{L^{2}(\mathbb{R})}, \\
            \|f_{\epsilon}^{\text{crit}}\|_{\dot{H}^{1}(S^{1})}&=\beta^{-\frac{1}{2}}\epsilon^{p+\frac{1-p}{2}}\|f^{\prime}\|_{L^{2}(\mathbb{R})}, \\
            \|f_{\epsilon}^{\text{crit}}\|_{\dot{H}^{2}(S^{1})}&=\beta^{-\frac{3}{2}}\epsilon^{-1+2p+\frac{1-p}{2}}\|f^{\prime\prime}\|_{L^{2}(\mathbb{R})}.
        \end{align*}
        From the last norm, we get the condition that
        \begin{equation*}
            -1+2p+\frac{1-p}{2}>0 \quad \iff \quad p>\frac{1}{3}.
        \end{equation*}
        Hence, for $\gamma=0$, choosing $p>\frac{1}{3}$, and in particular $p=\frac{1}{2}$ yields the desired result, as $(\mathring{R}_+)_n=f_{\frac{1}{n}}^{\text{sub}}$ satisfies \eqref{eq:data} and
        \begin{equation*}
            \lim_{n\to \infty } \|(\mathring{R}_+)_n\|_{H^{2}}=0.
        \end{equation*}
    
    \end{proof}

    \section{Proof of Theorem \ref{energyest}}\label{sec:thm2}

    In this section we discuss how a bound for the minimal time of existence can be obtained as a corollary of the results \cite{fajman2025arxiv} via energy methods. 

    \begin{remark}
        For the sake of coherence, we will state the contents of this chapter for plane-symmetric solutions. Note, however, that the argument adapts verbatim to the $3+1$-dimensional non-symmetric case, the only difference being the regularity of the data and hence the solution. Inspecting \cite[Lemma 4.17]{fajman2025arxiv}, we see that the only difference in the evolution of the energy compared to the plane-symmetric case is the appearance of another negative definite term featuring the curl of $v$.
    \end{remark}

\begin{lem}\label{lem:energyid}
    Suppose $s>\frac{3}{2}$ and $U=(L,v)$ solves the \eqref{system} on the time interval $[t_{0},t_{1})$. Then, there exists a non-negative functional $E[\cdot,\cdot]$ that satisfies the inequality
    \begin{equation}\label{eq:energyid}
        \frac{d}{dt}E[L,v]\leq -t^{-1}\alpha (1-3K)E[L,v]+Ct^{-2+\alpha}E[L,v]+\left(t^{-1}+t^{-\alpha }\right)\beta(E[L,v]),
    \end{equation}
    where  $|\beta(x)|=\mathcal{O}(x^{\frac{3}{2}})$, provided $x\leq 1$, and $C>0$. In addition, there exists a constant $\epsilon_{0}>0$, such that, given that $\frac{1}{2}(\|\partial f\|_{H^{s-1}}^{2}+\|g\|_{H^{s}}^{2})\leq \epsilon$ for some $0<\epsilon\leq \epsilon_{0}$, 
    \begin{equation*}
        E[f,g] \simeq \|\partial f\|_{H^{s-1}}^{2}+\| g\|_{H^{s}}^{2}.
    \end{equation*}
    We define the semi-norm $\|\cdot\|_{\tilde{H}^{s}}$ via $\|(f,g)\|_{\tilde{H}^{s}}^{2}=\|\partial f\|_{H^{s-1}}^{2}+\|g\|_{H^{s}}^{2}$. 
\end{lem}
    \begin{remark}
        The fact that $\|\cdot\|_{\tilde{H}^{s}}$ is only a semi-norm is not an issue. The reason is that the equations of motion \eqref{system} do not contain $L$ in its undifferentiated form. Hence, $L$ stays bounded as long as the sources of its evolution are bounded. This is accomplished globally in \cite{fajman2025arxiv} via a bootstrap. 
    \end{remark}

    \begin{proof}
        This Lemma is an corollary of \cite[Lemma 4.1]{fajman2025arxiv}, \cite[Lemma 4.17]{fajman2025arxiv}, and \cite[Lemma 4.19]{fajman2025arxiv}. For direct treatment of the plane-symmetric case, we refer to \cite{fajman2025cqg}.
    \end{proof}
    We will now use the energy inequality in \eqref{eq:energyid} to derive a lower bound on the lifespan of any given solution. 

\begin{proof}[Proof of Theorem \ref{energyest}]
    We focus on the subcritical case. The argument in the critical case is analogous. 
    
    Suppose that 
    \begin{equation*}
        \|U_{0}\|_{\tilde{H}^{s}}\leq \epsilon,
    \end{equation*}
    Now, we define the energy functional $\bar{E}=t^{\alpha (1-3K)}E$. Furthermore, we assume that $\epsilon_{0}$ is sufficiently small so that Lemma \ref{lem:energyid} applies, $E[U_{0}]<1$. We then have
    \begin{equation*}
        \bar{E}(t_0)=t_0^{\alpha(1-3K)}E(t_0)\leq c\epsilon^{2},
    \end{equation*}
    where $c>0$ is a constant. Without loss of generality, we redefine $\epsilon$ so that $\bar{E}(t_0)\leq \epsilon^2$. Inspecting \eqref{eq:energyid}, this can certainly always be achieved, since by \eqref{eq:energyid}, we find that
    \begin{equation*}
        \frac{d}{dt}\bar{E}\lesssim t^{-2+\alpha}\bar{E}+\left(t^{-1-\frac{1}{2}\alpha(1-3K)}+t^{-\frac{3}{2}\alpha (1-K)}\right)\bar{E}^{\frac{3}{2}}\lesssim t^{-2+\alpha}\bar{E}+ t^{-\frac{3}{2}\alpha (1-K)}\bar{E}^{\frac{3}{2}}.
    \end{equation*}
    Now, let $t_{*}>t_{0}$ be the first time at which 
    \begin{equation*}
        \bar{E}[U(t_*)]=A^2\epsilon^2,
    \end{equation*}
    where $A$ is a large constant to be chosen later on. Depending on $A$ we choose $\epsilon_0$ sufficiently small, such that Lemma \ref{lem:energyid} remains valid for $t\in[t_0,t_*]$. Then, necessarily, for $t\in[t_0,t_{*})$
    \begin{equation*}
        \frac{d}{dt}\bar{E}(t)\lesssim t^{-2+\alpha}\bar{E}+ t^{-\frac{3}{2}\alpha (1-K)}\bar{E}^{\frac{3}{2}}\Rightarrow -\frac{d}{dt}\left(\sqrt{\bar{E}}\right)^{-1}\leq B t^{-2+\alpha}\left(\sqrt{\bar{E}}\right)^{-1}+ Bt^{-\frac{3}{2}\alpha (1-K)}.
    \end{equation*}
    for some sufficiently large constant $B>0$. Setting $D(t)=\left(\sqrt{\bar{E}(t)}\right)^{-1}\exp(\int_{t_0}^{t}Bs^{-2+\alpha }ds)$ we find
    \begin{equation}\label{eq:Dineq}
        \frac{d}{dt}D(t)\geq  D(t)Bt^{-2+\alpha}-Bt^{-2+\alpha}D(t)-Bt^{-\frac{3}{2}\alpha(1-K)}e^{\int_{t_0}^{t}Bs^{-2+\alpha }ds}\geq -BCt^{-\frac{3}{2}\alpha(1-K)},
    \end{equation}
    where 
    \begin{equation*}
       C=\exp (B\int_{t_0}^{\infty}t^{-2+\alpha}dt)
   \end{equation*}
   Integrating the inequality \eqref{eq:Dineq} yields
   \begin{equation*}
       -\left(\sqrt{\bar{E}}\right)^{-1}(t_*)\exp\left(B\int_{t_0}^{t_*}t^{-2+\alpha}dt\right)+\left(\sqrt{\bar{E}}\right)^{-1}(t_0)\leq BC\frac{t_*^{1-\frac{3}{2}\alpha (1-K)}-t_0^{1-\frac{3}{2}\alpha (1-K)}}{1-\frac{3}{2}\alpha(1-K)},
   \end{equation*}
    Considering that $r=1-\frac{3}{2}\alpha (1-K)>0$ as well as $-2+\alpha<-1$, it follows that
    \begin{align*}
       t_*^{1-\frac{3}{2}\alpha (1-K)}&\geq-r\left(BC\sqrt{\bar{E}}\right)^{-1}(t_*)\exp\left(B\int_{t_0}^{t_*}t^{-2+\alpha}dt\right)+r\left(BC\sqrt{\bar{E}}\right)^{-1}(t_0) +t_0^{1-\frac{3}{2}\alpha (1-K)}\\
       &\geq -\frac{r}{ABC\epsilon}\exp\left(B\int_{t_0}^{t_*}t^{-2+\alpha}dt\right)+\frac{r}{BC\epsilon}. 
   \end{align*}
   Choosing
   \begin{equation*}
       A>\exp (B\int_{t_0}^{\infty}t^{-2+\alpha}dt)
   \end{equation*}
   yields $t_*^{1-\frac{3}{2}\alpha (1-K)}\gtrsim \frac{1}{\epsilon}$. Alternatively, the threshold is never reached and the solution is global via the continuation principle and the fact that $L$ remains bounded. 
\end{proof}

    \clearpage
    
    \appendix 

    \section{ODE theory}

    In this section we state an prove two lemmata concerning ODE theory. These are versions of standard results that can be found, e.g., in \cite{hoermander1997}, adapted for the specific purposes of this article.  
    
    \begin{lem}[ODE Blowup Lemma]\label{lem:blowup}
		Let $ A_{0,1,2} \in C([t_{0},t_{*}]) $ and $ A_{2}(t)>0 $ for all $ t\in[t_{0},t_{*}] $. Let Furthermore
		\begin{equation*}
			K_{2}(t^\prime)=\int_{t_{0}}^{t^\prime}e^{K_{1}(t)}A_{2}(t)dt, \quad K_{1}(t^\prime)=\int_{t_{0}}^{t^\prime}A_{1}(t)dt,\quad K_{0}(t^\prime)=\int_{t_{0}}^{t^\prime}e^{-K_{1}(t)}|A_{0}(t)|dt.
		\end{equation*}
		If 
        \begin{equation}\label{eq:condinit}
            v(t_0)=v_0 \qquad v_0>K_0(t_*),
        \end{equation}
        and
		\begin{equation}\label{eq:cond}
			v_{0}\geq (K_{2}(t_*)^{-1}+K_{0}(t_*))
		\end{equation}
		the ODE 
		\begin{equation*}
			\frac{d}{dt}v(t)=A_{2}(t)v(t)^{2}+A_{1}(t)v(t)+A_{0}(t)
		\end{equation*}
		does not have a bounded solution on $[t_{0},t_{*})$.
	\end{lem} 
    \begin{proof}
        Assume that we have a bounded solution $v$ on the interval $[t_0,t_*)$. We define the function $V$ via 
        \begin{equation*}
            V(t)\coloneqq e^{-K_1(t)}v(t). 
        \end{equation*}
        Then
        \begin{equation}\label{eq:odeV}
        \begin{aligned}
            \frac{d}{dt}V(t)&=-A_1(t)V(t)+e^{-K_1(t)}\left(A_{2}(t)v(t)^{2}+A_{1}(t)v(t)+A_{0}(t)\right)\\
            &=\underbrace{e^{K_{1}(t)}A_{2}(t)}_{\eqqcolon \tilde{A}_2(t)}V(t)^2+\underbrace{e^{-K_{1}(t)}A_{0}(t)}_{\eqqcolon \tilde{A}_0(t)},
        \end{aligned}
        \end{equation}
        and $V(t_0)=v_0$. Now consider the auxiliary ODE
        \begin{equation}\label{eq:auxode}
            \frac{d}{dt}W(t)=\tilde{A}_2(t)\left(W(t)-K_0(t_*)\right)^2, \qquad W(t_0)=v_0.
        \end{equation}
        A priori, we do not know whether a solution $W$ of \eqref{eq:auxode} exists on $[t_0,t_*)$. Certainly, $W$ exists locally on some maximal time interval $[t_0,t_1)$, where $t_0<t_1\leq t_*$. Then for $t\in[t_0,t_1)$, solving \eqref{eq:auxode} gives
        \begin{equation}\label{eq:locestW}
            (W(t)-K_0(t_*))^{-1}-(v_0-K_0(t_*))^{-1}=-K_2(t).
        \end{equation}
        Using \eqref{eq:condinit} and the fact that $K_2(t)$ is increasing for all $t\geq t_0$, we infer that $W$ is increasing. 
        
        Now, for $t\in[t_0,t_1)$, consider
        \begin{equation}\label{eq:auxest}
            \frac{d}{dt}\left(W(t)-K_0(t)\right)=\tilde{A}_2(t)\left(W(t)-K_0(t_*)\right)^2-|\tilde{A}_0(t)|\leq \tilde{A}_2(t)\left(W(t)-K_0(t)\right)^2+\tilde{A}_0(t),
        \end{equation}
        where we have used the fact that $K_0$ is non-negative and increasing, $W(t_0)=v_0>K_0(t_*)$, and $W$ is increasing. Comparing \eqref{eq:odeV} and \eqref{eq:auxest} and acknowledging the fact that $(W-K_0)(t_0)=v_0=V(t_0)$, we find that for all $t\in [t_0,t_1)$, $W(t)-K_0(t)\leq V(t)$. Since, by assumption, $V$ is bounded on $[t_0,t_*)$ and so is $K_0$, $W$ must be bounded as well. Hence, the solution $W$ extends to $t_1=t_*$. But then, taking the limit $t\to t_1=t_*$ of \eqref{eq:locestW}, and using the fact that $W$ is bounded on $[t_0,t_*)$, we see that
        \begin{equation}\label{eq:boundcontra}
            K_2(t_1)<  (v_0-K_0(t_*))^{-1},
        \end{equation}
        which contradicts \eqref{eq:cond}.
    \end{proof}
    \begin{lem}\label{lem:riccatiexist}
        Let $ A_{0,1,2} \in C([t_{0},t_{*}]) $ and $ A_{2}(t)>0 $ for all $ t\in[t_{0},t_{*}] $. Furthermore, let $K_0$, $K_1$ and $K_2$ be defined as in Lemma \ref{lem:blowup}. 
        
		If $v_0\in \mathbb{R}_+$ satisfies
		\begin{equation}\label{eq:condexist}
			v_{0} <  (K_{2}(t_*)^{-1}-K_{0}(t_*))
		\end{equation}
		the ODE 
		\begin{equation*}
			\frac{d}{dt}v(t)=A_{2}(t)v(t)^{2}+A_{1}(t)v(t)+A_{0}(t)
		\end{equation*}
		has a solution solution on $[t_{0},t_{*}]$ with $v(t_0)=v_0$.
    \end{lem}
    \begin{proof}
        Adopting the notation of the proof of Lemma \ref{lem:blowup}, we now consider the auxiliary ODE
        \begin{equation}\label{eq:auxode2}
            \frac{d}{dt}W(t)=\tilde{A}_2(t)\left(W(t)+K_0(t_*)\right)^2, \qquad W(t_0)=v_0.
        \end{equation}
        Solving \eqref{eq:auxode2} locally gives 
        \begin{equation*}
            \left(W(t)+K_0(t_*)\right)^{-1}=\left(v_0+K_0(t_*)\right)^{-1}-K_2(t).
        \end{equation*}
        By \eqref{eq:condexist} and the fact that $W$ is increasing, we see that the right-hand side is bounded from uniformly below for all $t\in[t_0,t_*]$ and hence, solution can be extended up until $t_*$. We calculate 
        \begin{equation}\label{eq:Westbelow}
            \frac{d}{dt}\left(W(t)+K_0(t)\right)=\tilde{A}_2(t)\left(W(t)+K_0(t_*)\right)^2+|\tilde{A}_0(t)|\geq \tilde{A}_2(t)\left(W(t)+K_0(t)\right)^2+\tilde{A}_0(t).
        \end{equation}
        Now, comparing \eqref{eq:Westbelow} with \eqref{eq:odeV}, and using the fact that $(W+K_0)(t_0)=v_0=V(t_0)$, we conclude that $V(t)\leq W(t)+K_0(t)\leq W(t)+K_0(t_*)$ as long as $V$ exists and $t\leq t_*$. Considering that
        \begin{equation*}
            \frac{d}{dt}V(t)\geq -|\tilde{A}_0(t)|, 
        \end{equation*}
        we gain the lower bound $V(t)\geq v_0-K_0(t)$. Hence, $V$ can be extended up until $t_*$. 
    \end{proof}

    \clearpage

    \printbibliography
    
\end{document}